\documentclass{amsart}

\usepackage{amsmath, amsfonts, amsthm,verbatim,graphicx, hyperref}

\usepackage{tikz}

\usetikzlibrary{matrix,arrows}

\usetikzlibrary{arrows,decorations.markings}

\tikzset{vertex/.style={circle,draw,fill,inner sep=0pt,minimum size=1mm}}
\tikzset{oriented edge/.style={thick, postaction={decorate,
         decoration={markings,mark=at position .5 with \arrow{angle 90};}}}}

\newtheorem{thm}{Theorem}
\newtheorem{lem}[thm]{Lemma}

\theoremstyle{definition}
\newtheorem{exl}[thm]{Example}
\newtheorem{defn}[thm]{Definition}
\newtheorem{rem}[thm]{Remark}

\newcommand{\lift}{\widetilde}
\newcommand{\llangle}{\langle \langle}
\newcommand{\rrangle}{\rangle \rangle}

\newcommand{\Z}{\mathbb Z}
\newcommand{\Q}{\mathbb Q}
\DeclareMathOperator{\Fix}{Fix}

\begin{document}

\title [Generalizing the rotation interval to vertex maps on graphs] {Generalizing the rotation interval to vertex maps on graphs}

\author[Bernhardt]{Chris Bernhardt}

\address{Department of mathematics and computer science\\Fairfield
University\\Fairfield\\CT 06824}
\email{cbernhardt@fairfield.edu}

 \author[Staecker]{P. Christopher Staecker}
 \address{Department of mathematics and computer science\\Fairfield
University\\Fairfield\\CT 06824}
\email{cstaecker@fairfield.edu}

\subjclass[2000]{37E15, 37E25, 37E45}

\keywords{rotation number, rotation interval, graphs, vertex maps, periodic orbits, homotopic to constant map, $\Q$-group}

\begin{abstract}
Graph maps that are homotopic to the identity and that permute the vertices are studied. Given a periodic point for such a map, a {\em rotation element} is defined in terms of the fundamental group. A number of results are proved about the rotation elements associated to periodic points in a given edge of the graph. Most of the results show that the existence of two periodic points with certain rotation elements will imply an infinite family of other periodic points with related rotation elements. These results for periodic points can be considered as generalizations of the rotation interval for degree one maps of the circle.
\end{abstract}

\maketitle

\section{Introduction}
Poincar\'e introduced the rotation number for a homeomorphism of the circle while studying the precession of the perihelion of planetary orbits \cite{P}. This idea was extended to degree one maps of the circle by Newhouse, Palis and Takens \cite{NPT}. The ideas have since been extended further to maps, homotopic to the identity,  of the annulus and torus, for example see \cite{F, L, M}. For a good history of the ideas and generalizations see \cite{ALM}. Most of the generalizations used homological ideas rather than non-commutative homotopical ones, but an approach using homotopy is given in \cite{BM}.

In this paper we consider rotations of periodic points for maps on graphs that are homotopic to the identity. First consider maps of the circle. Let $a$ be the generator of the fundamental group of $S^1$. Suppose $f: S^1 \to S^1$ is a degree one map (i.e. a map homotopic to the identity) that has a periodic point $x$ with period $n$. The homotopy gives a path from $x$ to $f^n(x)$. This path is a closed loop and so must correspond to wrapping around the circle a certain number, say $m$, times. Thus this path can be associated to $a^m \in \pi_1$. The rotation number associated to $x$ is the average rotation around the circle $m/n$. 

In the more general case where the underlying graph is not topologically a circle, the same construction can be performed. The fundamental group of the graph will be a free group on a number of generators. If $x$ is a periodic point of $f$ with period $n$ and $f$ is homotopic to the identity then the homotopy will give a path from $x$ to $f^n(x)$ that corresponds to an element $w \in \pi_1$. We define the {\em rotation element}, $\rho(x)$, to be $w^{1/n}$. So for maps of the circle the rotation number is just the exponent of the rotation element.

In this paper the maps $f$ from a graph $G$ to itself are taken to be homotopic to the identity. They will also be assumed to belong to the class of maps that permute the vertices. These maps have been studied in \cite{B, B2, B3} where they are called {\em vertex maps}. 

In \cite{B} rotation matrices for vertex maps homotopic to the identity were defined. These matrices give homological information. When the vertices of the graph form one periodic orbit, it is shown in \cite{B} that the rotation matrix is the zero matrix. This tells us that the matrix cannot distinguish any point whose rotation element is different from the rotation element of the vertices. In this situation it is clear that our approach using homotopy gives more information than one using homology. This is explained more fully in the last section.

The general spirit of our results is to generalize the following fact about rotation numbers on the circle: For degree one maps of the circle, if there are points with rational rotation numbers $r, q$, then for any $p$ with $r \le p \le q$, there is a point with rotation number $p$. Our main results will always have this flavor, showing that the existence of two periodic points with certain rotation elements will imply infinitely many additional periodic points with rotation elements related to the original two.

The outline of the paper is as follows: In Section \ref{prelimsection} we describe a labeling scheme for the universal cover, and define the rotation element in terms of a lift. As in the circle case, different liftings of the map $f$ give different rotation elements. The effect of changing the lifting is described. After the preliminary definitions we look at the cases where we are given periodic points $x$ and $y$ with rotation elements $\rho(x)$ and $\rho(y)$ and ask what other periodic points and rotation elements the underlying map must have. In Section \ref{interiorsection} we consider the case when $x$ and $y$ are in the interior of an edge. Then in Sections \ref{fpsection} and \ref{belongsection} we consider the case when $x$ and $y$ are vertices of an edge.


\section{Preliminaries and definition of the rotation element}\label{prelimsection}
Throughout, $G$ will be a finite directed graph with some vertex $x_0 \in G$ chosen as a base point for the fundamental group $\pi = \pi_1(G,x_0)$. In our graphs we will allow multiple edges  between any two vertices, but for simplicity we do not allow looped edges.

Let $\lift G$ be the universal covering space of $G$, which is an infinite graph. We will view elements of $\pi$ interchangeably as loop classes in $G$ as well as covering transformations on $\lift G$.
We label the vertices of $\lift G$ by choosing one particular lifted copy $\lift V_i$ of each vertex $V_i\in G$, and denoting the other vertices of $\lift G$ by covering transformations applied to the various $\lift V_i$. We will label the edges of $\lift G$ similarly: choose a particular lifting $\lift E_j$ for each edge $E_j$. All other edges of $\lift G$ projecting on to $E_j$ will be denoted $g \lift E_j$ for various $g \in \pi$, where the group element $g$ is viewed as a covering transformation on $\lift G$.

With the above choices, there is a specific way to lift any point $x\in G$ to some $\lift x \in \lift G$. If $x$ is a vertex we have already chosen its lift. If $x$ is on the interior of some edge $E$, then we choose $\lift x$ to be the single point in $p^{-1}(x) \cap \lift E$, where $\lift E$ is the chosen lift of the edge $E$. Throughout the paper we will use the tilde to indicate these specific choices of lifted points and edges. We will see that alternative choices for these lifts will not meaningfully affect the concepts defined below.

One major convenience of the above labeling scheme is that when $\gamma \in \pi$, we will have $\lift f(\gamma \lift x) = f_\#(\gamma)\lift f(\lift x)$, where $\lift f_\#:\pi \to \pi$ is the induced homomorphism on the fundamental group. When $f$ is homotopic to the identity, as will always be the case for the later sections, this means that $\lift f(\gamma \lift x) = \gamma \lift f(\lift x)$. 

\begin{rem}\label{piordering}
Standard notation for group elements and function composition leads to some unavoidable confusion when we consider fundamental group elements as deck transformations. As usual, fundamental group elements are read from left to right, so the element $ab\in \pi$ represents the loop $a$ followed by the loop $b$. Thus for $a,b\in \pi$, the point denoted $ab \, \lift x$ must be written as $b ( a (\lift x))$ when we view $a$ and $b$ as covering transformations. 
\end{rem}

\subsection{Definition of the rotation element}

\begin{defn}\label{roteltdefn}
Let $x \in \Fix(f^n)$. This means there is some word $w\in \pi$ with $\lift f^n(\lift x) = w\lift x$. We define the \emph{rotation element} of $x$ (with respect to the lift $\lift f$) to be $\rho(x) = w^{1/n}$. 
\end{defn}

The rotation element $\rho(x)$ of a point will be some rational power of an element of $\pi$. When $\pi$ is a free group, $\rho(x)$ is an element of a ``free group with rational exponents'' known as a \emph{free $\Q$-group}. These groups are the completion of free groups with respect to roots, first constructed by Baumslag (then called free $\mathcal D$-groups) in \cite{baum65}, and are still actively studied today. We will not require much sophistication in our use of these $\Q$-groups, since our rotation elements always consist of a single element of $\pi$ to some rational power. (We never consider elements of the form $w^{1/n}v^{1/m}$.)

The exponents are treated formally, modulo the property that $(w^k)^r = w^{kr}$ for an integer $k$ and rational $r$. The following lemma shows that, using this exponentiation, the rotation element is well defined with respect to the specific choice of $n$, which may or may not be the minimal period of $x$.

\begin{lem}
Let $x \in \Fix(f^n)$ and $k\in \Z$. Then the rotation element $\rho(x)$ is independent of whether we view $x$ as a periodic point of period $n$ or of period $kn$.
\end{lem}
\begin{proof}
Viewing $x\in \Fix(f^n)$, say that $\rho(x)=w^{1/n}$. This means that $\lift f^n(\lift x) = w\lift x$. Thus we have $\lift f^{kn}(\lift x) = w^k \lift x$, and so viewing $x\in \Fix(f^{kn})$ the rotation element is $(w^k)^{1/kn} = w^{1/n}$ as desired.
\end{proof}

Definition \ref{roteltdefn} implicitly uses a specific choice of the point $\lift x \in p^{-1}(x)$ as well as a specific lifting $\lift f$ of the map $f$. Alternative choices will change the rotation element, but in predictable ways, as the following lemmas show. In the first lemma we show the rotation element is independent up to conjugacy of the choice of lifted point. By ``up to conjugacy'' we mean that the rotation elements in question can be written as $w^r$ and $v^r$, where $w$ and $v$ are conjugate in $\pi$.

\begin{lem}\label{pointindep}
When $f$ is homotopic to the identity, the rotation element is independent of the choice of the lifted point $\lift x\in p^{-1}(x)$ up to conjugacy.
\end{lem}
\begin{proof}
Let $\lift f^n(\lift x) = w\lift x$, so the rotation element is $w^{1/n}$ when we use the lifted point $\lift x$. Any alternative choice of a point of $p^{-1}(x)$ will have the form $\gamma \lift x$ for some $\gamma \in \pi$. Since $f$ is homotopic to the identity we have 
\[ \lift f^n (\gamma \lift x) = f^n_\#(\gamma) \lift f^n(\lift x) =  f^n_\#(\gamma) w \lift x = (\gamma w \gamma^{-1}) \gamma \lift x,\]
and so the rotation element when using the lifted point $\gamma \lift x$ is $(\gamma w \gamma^{-1})^{1/n}$, which is conjugate to $\rho(\lift x) = w^{1/n}$.
\end{proof}

The definition of the rotation element does depend on the choice of lift $\lift f$, but in a predictable way: 
\begin{lem}\label{liftindep}
Let $\rho(x)=w^{1/n}$ with respect to some lift $\lift f$ and let $\gamma \in \pi$. If $f$ is homotopic to the identity, then $x$ has rotation element $(\gamma^n w)^{1/n}$ with respect to the lift $\gamma \lift f$.
\end{lem}
\begin{proof}
We have $(\gamma\lift f)^n(\lift x) = \gamma^n \lift f^n(\lift x) = \gamma^n w\lift x$, and so the rotation element is $(\gamma^nw)^{1/n}$.
\end{proof}

If the fundamental group has a single generator, $a$, and $\rho(x)=a^{m/n}$ with respect to $\lift f$, then $\rho(x)=a^{k+(m/n)}$ with respect to the lifting $a^k \lift f$. This is in agreement with the classical theory on circle maps, where changing the lifting adds an integer to the rotation number.

The following lemma shows that, when $f$ is homotopic to the identity, the rotation element is the same up to conjugacy for the various points of a periodic orbit. 
\begin{lem}\label{orbitindep}
Let $x \in \Fix(f^n)$ and let $O \subset G$ be the orbit of $x$. When $f$ is homotopic to the identity, the rotation elements of points in $O$ are the same up to conjugacy.
\end{lem}
\begin{proof}
Let $y \in O$, so $y = f^m(x)$ for some $m$, and there is some $\gamma$ with $\gamma \lift f^m(\lift x) = \lift y$. Let $\lift f^n(\lift x) = w\lift x$, so $\rho(x)=w^n$. Then we have
\[ \lift f^n(\lift y) = \lift f^{n}(\gamma \lift f^m(\lift x)) = \gamma \lift f^{n+m}(\lift x) = \gamma \lift f^m(w\lift x) = \gamma w \lift f^m(\lift x) = \gamma w \gamma^{-1} \lift y,
\]
and so $\rho(y) = (\gamma w \gamma^{-1})^{1/n}$ which is conjugate to $w^{1/n}$ as desired.
\end{proof}

\begin{rem}
When $f$ is not homotopic to the identity, the above lemmas will still hold with some simple modifications involving the induced homomorphism $f_\#:\pi \to \pi$. We also must use ``twisted conjugacy'' classes rather than conjugacy classes. We say that $w$ and $v$ are $f_\#^n$-twisted conjugate when there is some $\gamma$ with $w = f^n_\#(\gamma)v\gamma^{-1}$. This is the standard approach in Nielsen fixed and periodic point theory, in which fixed points are grouped according to their Reidemeister (twisted conjugacy) classes. See \cite{jian83} for details. The proof of Lemma \ref{pointindep} will hold in this setting without modifications.

For Lemma \ref{liftindep} when $f$ is not homotopic to the identity, the factor of $\gamma^n$ needs to be replaced by $\gamma f_\#(\gamma) \dots f_\#^{n-1}(\gamma)$, which is the ``algebraic boost'' also familiar in Nielsen theory. 

As is well known in Nielsen theory, the points of a periodic orbit do not all have the same Reidemeister classes. Nonetheless, it is well known (and proven by the same argument used in Lemma \ref{orbitindep}) that if $x$ and $y$ are as in the lemma and $\rho(x)=w^{1/n}$, then $\rho(y)$ is $f^n_\#$-twisted conjugate to $f^m_\#(w)^{1/n}$. The Reidemeister classes of points in the same orbit can then be grouped into a well defined ``Reidemeister orbit'' which is fundamental in Nielsen periodic point theory. 
\end{rem}

\subsection{Coherent labelings of the universal cover}
Thus far we have allowed the labeling of the vertices and edges of $\lift G$ to be somewhat arbitrary: for example for an edge $E$ with initial vertex $V$, we have not required that the particular chosen lift $\lift E$ have initial vertex $\lift V$. In order to prove meaningful theorems about the rotation element, we must require a more meaningful scheme for labeling the vertices and edges of $\lift G$. The general idea is that that adjacent vertices and edges in $\lift G$ will whenever possible have the same $\pi$-labels.

This is achieved specifically as follows: Choose a spanning tree $\mathcal T$ of $G$ with root at the basepoint vertex $x_0$, and lift the entire spanning tree $\mathcal T$ into a tree of $\lift G$ by lifting the root vertex to a particular arbitrarily chosen $\lift x_0$. This lifting of $\mathcal T$ is used to define the particular lifted copies $\lift V_i$ for every vertex of $G$, and also defines the particular liftings of each edge in $\mathcal T$. 

It remains to choose the particular lifting for each edge $E$ not in $\mathcal T$. Let $V_1$ and $V_2$ be the initial and terminal vertices of $E$. Because $\mathcal T$ is a spanning tree it will include $V_1$, and so we choose the particular lifting $\lift E$ to be the lift of $E$ starting at $\lift V_1$, the already chosen particular lifting of $V_1$. Note that $\lift E$ cannot connect $\lift V_1$ to $\lift V_2$, since this would close a loop in $\lift G$ which is simply connected. 

Thus $\lift E$ will connect $\lift V_1$ to $\theta \lift V_2$ for some $\theta \in \pi$. We will choose our labeling in a natural way so that this $\theta \in \pi$ is the fundamental group element corresponding to the single loop in $G$ obtained by adding $E$ to the chosen spanning tree $\mathcal T$. The set of $\theta$s constructed in this way for various $E \not \in \mathcal T$ are exactly the generators of $\pi$.

Generally for an edge $E_i$, let $\theta_i\in \pi$ be the unique word such that $\lift E_i$ connects $\lift V_1$ to $\theta_i\lift V_2$, where $V_1$ and $V_2$ are the vertices of $E_i$. These $\theta_i$ are trivial when $E_i\in \mathcal T$, and are generators of $\pi$ when $E_i \not \in \mathcal T$.

For the rest of the paper we will restrict our focus to labelings of $\lift G$ constructed as above, by lifting a spanning tree $\mathcal T$ and labeling terminal vertices of edges outside of $\mathcal T$ according to the corresponding generators of $\pi$. A labeling constructed in this way we will call \emph{coherent}.

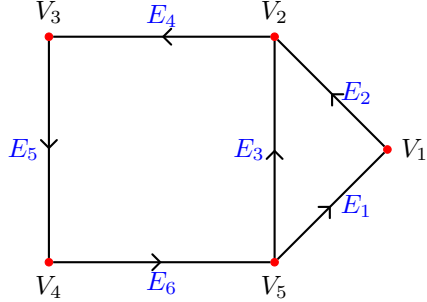
\begin{figure}
\begin{center}
	\begin{tikzpicture}[auto, scale=1.5]

		\node[vertex, red] (four)   at (0,0)    [label=below :$V_4$] {};
		\node[vertex, red] (five)   at (2,0)  [label=below:$V_5$] {};
		\node[vertex, red] (one) at (3,1)    [label=right:$V_1$] {};
		\node[vertex, red] (two) at (2,2)    [label=above:$V_2$] {};
		\node[vertex, red] (three) at (0,2)    [label=above:$V_3$] {};

		\draw[oriented edge] (four)   to node[below,blue]      {$E_6$} (five);
		\draw[oriented edge] (five)  to  node[midway,below,right, blue]{$E_1$}  (one);
		\draw[oriented edge] (five)   to node[blue]    {$E_3$} (two);
		\draw[oriented edge] (one)   to node[above, right, blue]    {$E_2$} (two);
		\draw[oriented edge] (two)   to node[above, blue]    {$E_4$} (three);
		\draw[oriented edge] (three)   to node[left, blue]    {$E_5$} (four);

	\end{tikzpicture}
\end{center}
\caption{Labeled graph\label{housegraph}}
\end{figure}

\begin{figure}
\begin{center}
	\begin{tikzpicture}[auto, scale=1.5]

		\node[vertex, red] (four)   at (0,0)    [label=below :$V_4$] {};
		\node[vertex, red] (five)   at (2,0)  [label=below:$V_5$] {};
		\node[vertex, red] (one) at (3,1)    [label=right:$V_1$] {};
		\node[vertex, red] (two) at (2,2)    [label=above:$V_2$] {};
		\node[vertex, red] (three) at (0,2)    [label=above:$V_3$] {};

		\draw[oriented edge, dashed] (four)   to  (five);
		\draw[oriented edge, very thick] (five)  to  node[midway,below,right, blue]{$E_1$}  (one);
		\draw[oriented edge, very thick] (five)   to node[blue]    {$E_3$} (two);
		\draw[oriented edge, dashed] (one)   to  (two);
		\draw[oriented edge, very thick] (two)   to node[above, blue]    {$E_4$} (three);
		\draw[oriented edge, very thick] (three)   to node[left, blue]    {$E_5$} (four);

	\end{tikzpicture}
	\caption{A spanning tree chosen for the graph from Figure 1\label{housest}}
\end{center}
\end{figure}
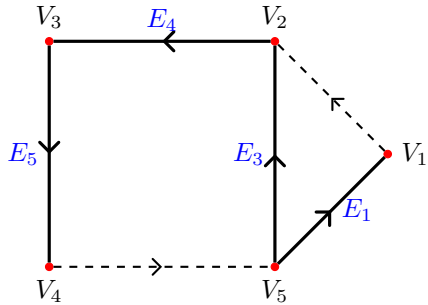

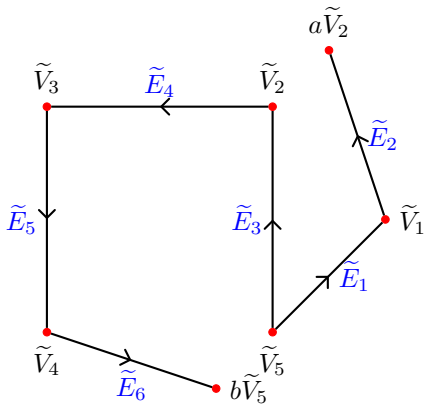
\begin{figure}
\begin{center}
	\begin{tikzpicture}[auto, scale=1.5]

		\node[vertex, red] (four)   at (0,0)    [label=below :$\lift V_4$] {};
		\node[vertex, red] (five)   at (2,0)  [label=below:$\lift V_5$] {};
		\node[vertex, red] (one) at (3,1)    [label=right:$\lift V_1$] {};
		\node[vertex, red] (two) at (2,2)    [label=above:$\lift V_2$] {};
		\node[vertex, red] (three) at (0,2)    [label=above:$\lift V_3$] {};
		\node[vertex, red] (atwo) at (2.5,2.5)    [label=above:$a\lift V_2$] {};
		\node[vertex, red] (bfive) at (1.5,-.5)    [label=right:$b\lift V_5$] {};

		\draw[oriented edge] (four)   to node[below,blue]      {$\lift E_6$} (bfive);
		\draw[oriented edge] (five)  to  node[midway,below,right, blue]{$\lift E_1$}  (one);
		\draw[oriented edge] (five)   to node[blue]    {$\lift E_3$} (two);
		\draw[oriented edge] (one)    to node[above, right, blue]    {$\lift E_2$} (atwo);
		\draw[oriented edge] (two)   to node[above, blue]    {$\lift E_4$} (three);
		\draw[oriented edge] (three)   to node[left, blue]    {$\lift E_5$} (four);

	\end{tikzpicture}

	\caption{Portion of the universal covering space, coherently labeled according to the spanning tree in Figure 2\label{houseuc}}
	
	\end{center}
	\end{figure}

As an example we give a construction of a coherent labeling when $G$ is the graph in Figure \ref{housegraph}. The vertex $V_1$ is chosen as our basepont, and as the root of the spanning tree shown in Figure \ref{housest}. In this example $\pi$ is free on two generators which we call $a$ and $b$, where $a$ is represented by the loop $E_2E_3^{-1}E_1$ and $b$ is represented by $E_1^{-1}E_3E_4E_5E_6E_1$. A portion of $\lift G$ is pictured in Figure \ref{houseuc}, with the coherent labeling induced by $\mathcal T$.
\subsection{The linearization of the map}

Each edge in the graph is homeomorphic to the unit interval. We use the homeomorphism to define the distance between points in an edge and to give each edge unit length. A path consisting of $m$ edges is defined to have length $m$ in the obvious way. Suppose that an edge $E_i$ is mapped by $f$ to a path with $m$ edges, then there is a natural induced map $f^*:[0,1] \to [0,m]$. We will say that $f$ is linear on $E_i$ if $f^*$ is linear.

We now define the {\em linearization} of the map $f$, which we will denote by $L_f$.  For all vertices $V \in G$, we define $L_f(V)=f(V)$. If $E_i$ is an edge with endpoints $V_1$ and $V_2$, we define $L_f$ to map $E_i$ linearly onto the contracted path from $f(V_1)$ to $f(V_2)$ that is obtained from $f(E_i)$.

More formally, let $[0,1]=I$, we define $L_f:G  \to G$ to be the {\em linearization} of $f$ if for each edge $E_i$ there is homotopy $h_i:E_i \times I \to G $ which has the following properties : $h_i(x,0)=f(x)$ for all $x \in E_i$; $h_i(x,1)=L_f(x)$ for all $x \in E_i$; $h_i(V_1, t)=f(V_1)=L_f(V_1)$ for all $t \in I$; $h_i(V_2, t)=f(V_2)=L_f(V_2)$ for all $t \in I$; and such that $L_f$ is linear on $E_i$.
If $L_f$ is the linearization of a map $f$, we say $L_f$ is \emph{linearized}.

In the literature the maps that we are calling linearized are sometimes referred to as {\em linear models} for tree maps or {\em connect-the-dots} maps for interval maps, see \cite{ ALM}.

\section{Rotation elements for interior points of an edge}\label{interiorsection}

Throughout the rest of the paper we consider only continuous vertex maps $f$ which are homotopic to the identity. Such maps behave predictably near the vertices according to the vertex permutation, but can be wild on the interior of the edges. For this reason we will not be able to derive strong results concerning the rotation elements on interior points based on information from the map $f$. 

However, the behavior on the interior of the edges is rigid for linearized maps, and this will allow more to be said. In fact information about rotation elements for the linearization of $f$ will provide information about rotation elements for $f$, as the following lemma shows. 

Our chosen lifting $\lift f$ defines canonically a lifting of the linearization $\lift L_f$ by lifting the homotopy from $f$ to $L_f$ to the universal cover starting at $\lift f$. For clarity below, let $\rho(x) = \rho_{\lift f}(x)$ be the rotation element of $x\in \Fix(f^n)$ with respect to $\lift f$, and let $\rho_{\lift L_f}(x)$ be the rotation element of $x\in \Fix(L_f^n)$ with respect to $\lift L_f$.

\begin{lem}\label{linearizationlem}
If $L_f$ has a periodic point $x$ with $\rho_{\lift L_f}(x)=w^{1/n}$, then $f$ has a periodic point $z$ with $\rho_{\lift f}(z)=w^{1/n}$
\end{lem}
\begin{proof}
It is clear that $\lift L_f( \lift E) \subseteq \lift f( \lift E)$. Induction gives that $\lift L_f^n( \lift E) \subseteq \lift f^n( \lift E)$. Since $\lift L_f^n(\lift x) = w\lift x$ we have $w\lift E \subset \lift L^n_f(\lift E)$. Thus we have $w \lift E \subset \lift f^n(\lift E)$.

Since $\lift f^n(\lift E)$ is a path in the simply connected space $\lift G$, the fact that $w\lift E \subset \lift f^n(\lift E)$ means that there is some interval $\lift J \subset \lift E$ with $\lift f^n(\lift J) = w\lift E$. Projecting into $G$ we have $f^n(J)=E$ and so there is some point $z\in J$ with $f^n(z)=z$ and $\lift f^n(\lift z) = w\lift z$ as desired.
\end{proof}

We will require a simple geometric lemma which holds for linear maps.

\begin{lem}\label{intervallem}
Let $E$ be an edge of $G$ and $g:G\to G$ be a linear vertex map.
If $x$ is in the interior of some edge $E$ and $g^n(x)=x$, then there is some interval $C \subset E$ containing $x$ with $g^n(C) = E$.
\end{lem}
\begin{proof}
Let $E$ be parameterized as the interval $[0,1]$, and we identify the point $x$ with its coordinate so that $x\in [0,1]$. Since $g$ is linear and $g^n(x)=x$ there is some ``slope'' $k$ such that $g^n(x+\epsilon) = x + k\epsilon$ for sufficiently small $\epsilon$. In fact this holds as long as $0\le x+k\epsilon \le 1$. Then it is clear that there is some $\delta$ and $\delta'$, one positive and one negative, such that $g^n(x+\delta) = 0$ and $g^n(x+\delta') = 1$. Taking $C$ to be the closed subinterval of $E$ that has endpoints $x+\delta$ and $x+\delta^\prime$ satisfies the desired conclusion.
\end{proof}

For $\alpha, \beta \in \pi$, let $\llangle \alpha, \beta \rrangle \subset \pi$ be the subset of elements generated by products of positive powers of $\alpha$ and $\beta$. This is the semigroup generated by $\alpha$ and $\beta$. Every element $w \in \llangle \alpha, \beta \rrangle$ can be written as $w = v_1 \dots v_l$ where each $v_i \in \{\alpha, \beta\}$ and $l\ge 1$. For $w_1, w_2 \in \pi$, and integers $m,n$, let $\mathcal S(w_1^{1/m}, w_2^{1/n})$ be the set of all elements of the form $v^{1/(rm+sn)}$ where $v\in \llangle w_1, w_2 \rrangle$ is built from $r$ occurances of $w_1$ and $s$ occurances of $w_2$.

In the case of maps on the circle, say $w_1=a^p$ and $w_2=a^q$ for integers $p,q$. Then the set $\mathcal S(a^{p/m}, a^{q/n})$ is the set of all words of the form $a^{(rp+sq)/(rm+sn)}$, which is the set of all $a^r$, where $r$ is any rational between $p/m$ and $q/n$. Thus the set $\mathcal S(w_1^{1/m}, w_2^{1/n})$ is a generalization of the interval of rotation numbers between those of $x$ and $y$.

\begin{thm}\label{bothinterior}
Suppose that $L_f$ has periodic points $x\in \Fix(L_f^m)$ and $y\in \Fix(L_f^n)$ in the interior of some common edge $E$. Then for any $w^{1/k} \in \mathcal S(\rho_{\lift L_f}(x),\rho_{\lift L_f}(y))$, there is a periodic point $z$ of $f$ on the interior of $E$ with $\rho_{\lift f}(z) = w^{1/k}$.
\end{thm}

\begin{proof}
Let $\rho_{\lift L_f}(x) = \alpha^{1/m}$ and $\rho_{\lift L_f}(y)=\beta^{1/n}$. Since $w^{1/k} \in \mathcal S(\rho(x),\rho(y))$, we will have $w^{1/k}=(v_1\dots v_l)^{1/(rm+sn)}$ where $v_i \in \{\alpha,\beta\}$ and $r$ and $s$ are the number of times that $\alpha$ and $\beta$ appear in this expression of $w$. We must show that there is a periodic point $t\in \Fix(f^{k})$ with $\lift f^{k}(\lift t) = w\lift t$. By Lemma \ref{linearizationlem} it suffices to show that there is some $z \in \Fix(L_f^{k})$ with $\lift L_f^{k}(\lift z) = w \lift z$.

Our proof is by induction on $l$. The case where $l=1$ is trivial. For the inductive case we assume that there is some periodic point $q$ with $\lift L_f^{k}(\lift q) = v_1\dots v_l \lift q$, and we will prove that there is some periodic point $z$ with period $p$ with $\lift L_f^p(\lift z) = v_1\dots v_l v_{l+1} \lift z$ for any $v_{l+1}\in \{\alpha, \beta\}$. The period $p$ should be $k+m$ or $k+n$ depending on whether $v_{l+1}$ is $\alpha$ or $\beta$. 

Without loss of generality we will assume $v_{l+1} = \alpha$, so we must show that there is a periodic point $z$ with $\lift L_f^{k+m}(\lift z) = v_1\dots v_{l}\alpha\lift z$.

Since $\lift L_f^n(\lift x) = \alpha \lift x$, by Lemma \ref{intervallem} there is some subinterval $C \subset E$ containing $x$ with $\lift L_f^m(\lift C) = \alpha \lift E$. Similarly there is a subinterval $D\subset E$ containing $q$ with $\lift L_f^{k}(\lift D) = v_1 \dots v_l \lift E$. Since $\lift L_f^m(\lift C) = \alpha \lift E$ we will have $ \alpha \lift D \subset L_f^m(\lift C)$, and so  
\[
\lift L_f^{k+m}(\lift C) = \lift L_f^{k}(\lift L_f^m(\lift C)) \supset \lift L_f^{k}(\alpha \lift D) 
= \alpha (\lift L_f^{k}(\lift D)) =  \alpha(v_1 \dots v_l \lift E) = v_1\dots v_l \alpha \lift E. 
\]
(The change in order in the last step is according to Remark \ref{piordering}.) 
Thus since $\lift C \subset \lift E$, there is a point $z \in C$ with $\lift L_f^{(k+m)}(\lift z) = v_1\dots v_l\alpha \lift z$ as desired.
\end{proof}




The proof of the theorem above can be extended to the situation where one periodic point is at a vertex, subject to an additional assumption on the behavior of $f$ near that vertex. The theorem is stated in terms of the initial vertex, and also holds (by the same argument) for the terminal vertex replacing ``begins'' with ``ends'' in the statement below. Note that $L_f$ and $f$ agree at the vertices of $G$, and so the distinction between $f$ and its linearization is not necessary at a vertex.
\begin{thm}
Let $x\in \Fix(f^m)$ be the initial vertex of $E$ and let $y\in \Fix(L_f^n)$ be a point on the interior of $E$. Also assume that the path $L_f^m(E)$ begins with $E$. Then for any $w^k \in \mathcal S(\rho_{\lift f}(x),\rho_{\lift L_f}(y))$, there is a periodic point $z$ of $f$ on the interior of $E$ with $\rho_{\lift f}(z) = w^k$.
\end{thm}
\begin{proof}
Since $L_f^m(E)$ begins with $E$, there will be some initial interval $A\subset E$ with $L_f^m(A) = E$. Since $L_f^m(\lift x) = \alpha \lift x$ and $\lift x$ is the initial endpoint of $\lift A$, we will have $\lift L_f^m(\lift A) = \alpha \lift E$.
Now the proof proceeds exactly as in Theorem \ref{bothinterior}, using $A$ in place of $C$.
\end{proof}

Note that the hypothesis to the above will always hold in particular if $x$ is a degree 1 vertex.

\section{Fixed points and the paths $\lift P_i$}\label{fpsection}
Given two periodic points $x$ and $y$ with rotation elements $\rho(x)$ and $\rho(y)$, it is natural to ask about the existence of other periodic points and rotation elements. The previous section looked at the case when $x$ and $y$ belong to the interior of an edge. In this and the next section we consider the case when $x$ and $y$ are the two endpoints of a given edge. We will not need to worry about the disctinction between $f$ and $L_f$ in these sections, since these functions agree on the vertices.

We will see in this section that we cannot in general expect an analogue of Theorem \ref{bothinterior} to hold when both $x$ and $y$ are vertices. In Example \ref{nointerval} we will show a situation where $x$ and $y$ are periodic points at vertices, but their edge contains only a single other periodic point (which is fixed). This is in contrast to the infinite set of periodic points which would  be implied by an analogue of Theorem \ref{bothinterior}.






In what follows we will choose some edge $E \in G$ with initial vertex $V_1$ and terminal vertex $V_2$. We will assume $V_1$ is a periodic point with period $m$ and rotation element $w_1^{1/m}$ and that $V_2$ is a periodic point with period $n$ and rotation element $w_2^{1/n}$. We will choose the edge $E$ to be part of the spanning tree of $G$ used to construct the coherent labeling. This means that  $\lift E$ has initial vertex $\lift V_1$ and terminal vertex  $\lift V_2$.

Given a positive integer $k$ we consider the path in $\lift G$ corresponding to $\lift f^{kmn} (\lift E)$. It is clear that this is a path that goes from $\lift f^{kmn} (\lift V_1)=w_1^{nk}\lift V_1$ to $\lift f^{kmn} ( \lift  V_2)= w_2^{mk} \lift  V_2$. The strategy in the following is to find edges in this path that project onto $E$ and then use the following lemma to deduce that periodic points exist with given rotation elements. 

\begin{lem}
Let $\lift E \in \lift G$. If the path  $\lift f^n (\lift E)$ contains the edge $\gamma \lift E$, then there exists a point $x \in E$ with rotation element $\gamma^{1/n}$.

\end{lem}

\begin{proof}
If the path  $\lift f^n (\lift E)$ contains the edge $\gamma \lift E$, then there must be a closed subinterval $\lift J$ of $\lift E$ with $\lift f^n(\lift J)=\gamma \lift E$. Projecting to the graph gives $f^n(J)=E$. Since $J$ is a closed subinterval contained in $E$, there must be a point $x \in E$ with $f^n(x)=x$. Then $\lift f^n (\lift x)=\gamma \lift x$ and so $\rho(x)=\gamma^{1/n}$.
\end{proof}

Since the covering space is a tree there is a unique shortest path between any two vertices. Given any path between vertices in the universal cover we say that the {\em contraction} of the path is the shortest path connecting the two vertices.

Let $\lift P_1$ denote the shortest path from $\lift V_1$ to $\lift w_1V_1$, and $\lift P_2$ denote the shortest path from $V_2$ to $w_2V_2$. More generally let $\lift P_1^k$ denote the shortest path from $\lift V_1$ to $w_1^k \lift V_1$, and $\lift P_2^k$ denote the shortest path from $ V_2$ to $w_2^k  V_2$. Let $\lift P_1^{-k}$ denote the reverse of $\lift P_1^k$. The following lemma follows immediately.

\begin{lem} 

The contraction of the path $\lift P_1^{-kn}\lift E \lift P_2^{km}$ is equal to the contraction of the path  $\lift f^{kmn} (\lift E)$.

\end{lem}

\begin{proof}
Both paths go from  $w_1^{nk}\lift V_1$ to $ w_2^{mk}\lift  V_2$.
\end{proof}

The previous two lemmas give the following which will be used throughout.

\begin{lem}\label{roteltsinpathlem}

If $\gamma \lift E$ is in the contraction of $\lift P_1^{-kn}\lift E \lift P_2^{km}$, then there exists an $x \in E$ with rotation element $\gamma^{1/kmn}$.
\end{lem}

We will say that a path begins with $\lift E$ if the initial edge in the path is either $\lift E$ or $\lift E^{-1}$. Thus saying that $\lift P_1$ begins with $\lift E$ means that the first edge is $\lift E$, but saying that $\lift P_2$ begins with $\lift E$ means that the first edge is $\lift E^{-1}$.

There are three general cases according to which of the paths $\lift P_i$ begin with $\lift E$, which we treat separately in three subsections.

\subsection{Neither or both of $\lift P_1$ nor $\lift P_2$ begin with $\lift E$.} 

\begin{thm} \label{fpthm}
If neither $\lift P_1$ nor $\lift P_2$ begins with $\lift E$, then $E$ contains a fixed point $x$ with $\rho(x) = 1$.
\end{thm}
\begin{proof}
 The contraction of $\lift P_1^{-n}\lift E \lift P_2^{m}$ must still contain $\lift E$. So $\lift f$ must have a fixed point $\lift x$  in $\lift E$. Thus $\lift f(\lift x) = \lift x$, and so $\rho(x) = 1^1$.
\end{proof}

\begin{thm} \label{fpboth}
If both $\lift P_1$ and $\lift P_2$ begin with $\lift E$, then $E$ contains a fixed point $x$ with $\rho(x)=1$.
\end{thm}

\begin{proof}
The first edge in $\lift P_1$ is $\lift E$ and the first edge of   $\lift P_2$  is $\lift E^{-1}$.
The contraction of $\lift P_1^{-n}\lift E \lift P_2^{m}$ must still contain $\lift E^{-1}$. So $\lift f$ must have a fixed point $\lift x$  in $\lift E$. 
\end{proof}

Note that the fixed points called $x$ in the two above theorems may not be on the interior of $E$. It is possible that one of $V_1$ or $V_2$ is fixed, and this is the only fixed point in $E$. As an example, consider the closed interval $[0,1]$ to be an edge, and the map $f$ restricted to this edge is given by $f(x)=x^2$. The two vertices given by $0$ and $1$ are fixed, neither $\lift P_1$ nor $\lift P_2$ contain $E$ (they are both trivial), but the only fixed points are the vertices.


\begin{exl}\label{nointerval}
Consider the graph given in Figure \ref{housegraph}, and a map with vertex permutation given by $(1,2,3,4,5)$. We will construct $f$ to be homotopic to the identity, and define $f$ according to the path the homotopy gives going from $V_i$ to $f(V_i)$: $V_1$ goes to $V_2$ along $E_2$. $V_2$ goes to $V_3$ along $E_4$. $V_3$ goes to $V_4$ along $E_5$. $V_4$ goes to $V_5$ along $E_6$. $V_5$ goes to $V_1$ along $E_1$. With the labeling as given in that section $\lift V_2$ has rotation element $(ba)^{1/5}$ and $\lift V_5$ has rotation element $(ab)^{1/5}$. 
It is also clear that the only periodic point in $E_3$ given by the linear map is a fixed point in the interior of $E_3$.
\end{exl}

\begin{exl}\label{newedge}
More generally, let $G$ be a graph that is not complete. Let $f:G\to G$ be a map that is homotopic to the identity. Find two vertices $V_i$ and $V_j$ with no edge connecting them. There is a path from $V_i$ to $f(V_i)$ given by the homotopy. Denote this path by $Q_i$. Let $Q_j$ be the corresponding path from $V_j$ to $f(V_j)$. Then construct a new edge $E$ from $V_i$ to $V_k$. Call this new graph $G'$. Define $f^\prime : G^\prime \to G^\prime$ by $f^\prime$ restricted to $G$ is $f$ and $f'(E)=Q_i^{-1}EQ_j$. This new map is homotopic to the identity. Neither  $\lift P_i^k$ nor $\lift P_j^k$ start with $\lift E$.
The only periodic points that have been added going from $f$ to $f^\prime$ are possible fixed points in $E$.
\end{exl}

The above shows that we cannot expect a simple generalization of Theorem \ref{bothinterior} when $x$ and $y$ are vertices of $E$. In these cases we have nontrivial rotation elements $\rho(x)$ and $\rho(y)$, but no rotation elements in $\mathcal S(\rho(x), \rho(y))$.

\subsection{Exactly one of $\lift P_1$, $\lift P_2$ begins with $\lift E$.}

Our techniques yield no useful results in this case. To illustrate, 
we give an example to show that there can be infinitely many periodic points with different rotation elements for this case.
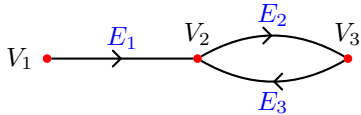
\begin{center}
	\begin{tikzpicture}[auto]

		\node[vertex, red] (one)   at (0,0)    [label=left :$V_1$] {};
		\node[vertex, red] (two)   at (2,0)  [label=above:$V_2$] {};
		\node[vertex, red] (three) at (4,0)    [label=above:$V_3$] {};

		\draw[oriented edge] (one)   to node[blue]      {$E_1$} (two);
		\draw[oriented edge] (two)  to  [bend left]  node[blue]     {$E_2$} (three);
		\draw[oriented edge] (three)   to [bend left] node[blue]       {$E_3$} (two);

	\end{tikzpicture}
\end{center}
\begin{exl}\label{notcont}
Consider the graph pictured above. The map $f$ interchanges $V_1$ and $V_3$ and fixes $V_2$. The map is homotopic to the identity and sends $V_1$ to $f(V_1)=V_3$ along the path $E_1E_2$. The homotopy gives $E_3E_1^{-1}$ as the path from $V_3$ to $f(V_3)=V_1$. Let $a$ be the generator of the fundamental group. The $\rho(V_1)=\rho(V_3)=a^{1/2}$ and $\rho(V_2)=a^0$. Notice that 
the path $\lift P_1$ begins with $\lift E_1$ and ends with the edge $a^k \lift E_1^{-1}$, and these are the only two edges that project onto $E_1$. The path $\lift P_2$ is the empty path. So the only edge in the contracted path $\lift P_1^{-k}\lift E_1 \lift P_2^k$ is $a^k \lift E_1^{-1}$. So in this case our techniques from this section tell us nothing about periodic points with rotation elements different from $(a^k)^{1/2k}=a^{1/2}$. 

However, we will see that edge $E_3$ satisfies the hypotheses of Theorem \ref{cfrotint} and so has periodic points with rotation elements of the form $a^q$ for all rationals satisfying $0< q < 1/2$. Since the image of every point in $E_3$ is in $E_1$, it must be the case that $E_1$ also has periodic points with rotation elements of the form $a^q$ for all rationals satisfying $0< q < 1/2$.
\end{exl}

\section{Rotation elements for vertices of an edge}\label{belongsection}
In this section we continue the theme of deducing certain periodic points and rotation elements based on the rotation elements of vertices of an edge. Unlike the previous section the theorems in this section will always imply infinitely many periodic points on the interior of the edge in question. So it is the case that the results of Section \ref{interiorsection} can be used to extend the results of this, since any two interior periodic points will imply still more periodic points by Theorem \ref{bothinterior}.

Our results in this section use a natural condition on an edge which detects whether or not that edge ``belongs'' to the loop representing certain fundamental group elements.

\subsection{Belonging}
Call an edge loop $P$ in $G$ \emph{cyclically reduced} when the first edge of $P$ is not equal to the reverse of the last edge. Every loop $P$ can be contracted onto a cyclically reduced loop by changing the basepoint to another vertex of $P$ and canceling adjacent inverse edges. We call the resulting loop a \emph{cyclic reduction} of $P$. 

\begin{defn}
Let $E$ be an edge of $G$ and $\alpha \in \pi$. We say that $E$ \emph{belongs} to $\alpha$ when every cyclic reduction of $\alpha$ uses the edge $E$ or $E^{-1}$.
\end{defn}

Consider as an example the graph in Figure \ref{housegraph} with $\pi = \langle a, b \rangle$ as described in that section. In this case $E_1, E_2, E_3$ all belong to $a$, and $E_3, E_4, E_5, E_6$ belong to $b$. For the graph used in Example \ref{notcont}, the edges $E_2$ and $E_3$ belong to $a$, but $E_1$ does not.

This intuitively clear definition of belonging is not very convenient for our purposes. We will typically use the following more technical statement.

\begin{lem}\label{belongslem}
Let $V$ be a vertex of the edge $E$, and let $w\in \pi$. If $E$ belongs to $w$, then there is some initial subword $\gamma$ of $w$ such that, for every $k\ge 0$, the contracted path from $w^k\lift V$ to $w^{k+1}\lift V$ uses the edge $w^j \gamma \lift E$ if and only if $j=k$.
\end{lem}
\begin{proof}
For simplicity, we prove the lemma in the case when $k=0$. Let $P$ be the cyclically reduced loop with base point $V$ represented by $w$, and let $\lift P$ be the lift of $P$ beginning at $\lift V$. This path is the fully contracted path from $\lift V$ to $w\lift V$, and so is the path referred to in the statement of the lemma. To obtain the general case when $k\neq 0$ we use the same argument as follows, but $\lift P$ should be defined as the lifting of $P$ beginning at $w^k\lift V$.

We will need to examine the edges used in the path $P$: let $P = E_1 \dots E_l$, and so we can write $\lift P = \lift E_1 (\sigma_2\lift E_2) \dots (\sigma_l \lift E_l)$. Our coherent labeling of the edges of $\lift G$ was constructed so that each edge $\lift E_i$ connects vertices of the form $\lift V_j$ to $\theta_i \lift V_k$ where $\theta_i \in \pi$ is a single letter or trivial. Those $\theta_i$ which are not trivial are all distinct. Thus each element $\sigma_i$ is obtained from $\sigma_{i-1}$ by possibly adding a single letter on the right. Since $\lift P$ is contracted this single added letter will never cancel with part of $\sigma_i$, since this would require $P$ to traverse an edge followed by its own inverse.

Thus the words $\sigma_i$ form a sequence of growing words, each term either equaling the previous or adding a single letter on the right. Since $\lift P$ ends at $w\lift V$, we can let $\sigma_{l+1} = w$ and still the sequence of $\sigma_i$ from $i=2$ to $i=l+1$ is a growing sequence of words, whose final word is $w$. Thus each $\sigma_i$ is an initial subword of $w$.

Since $E$ belongs to $w$, some particular edge $E_i$ must be the edge $E$, and then letting $\gamma = \sigma_i$ shows that $\lift P$ uses the edge $\gamma \lift E$  as desired. Note also that it is impossible for any $\sigma_l$ to equal $w^j\gamma$ for any $j>0$, since the $\sigma_l$ must all be initial subwords of $w$. The only initial subword of $w$ which can be written as $w^j\gamma$ is $\gamma$ itself. Thus $\lift P$ uses the edge $w^j\gamma\lift E$ if and only if $j=0$, as desired.
\end{proof}

The above can be stated in terms of the paths $\lift P_i$ as follows:
\begin{lem}\label{Pbelongslem}
Let $V_j$ be a vertex of $E$ and assume that $E$ belongs to $w\in \pi$. If $\rho(V_j)=w^{1/n}$ then there is some initial subword $\gamma$ of $w$ such that $\lift P_j^k$ contains the edge $w^i \gamma \lift E$ if and only if $0\le i < k$. 
\end{lem}
\begin{proof}
The path $\lift P_j^k$ is the contracted path from $\lift V_j$ to $w^k\lift V_j$, and so will be the concatenation of the paths from $w^i \lift V_j$ to $w^{i+1}\lift V_j$ for each $i$ with $0\le i < k$. By the above lemma there is some initial subword $\gamma_i$ such that the path from $w^i \lift V_j$ to $w^{i+1}\lift V_j$ uses the edge $w^j\gamma_i \lift E$ if and only if $j=i$. Following the construction in the lemma above we see that all these $\gamma_i$ are the same, since the sequence of letters called $\theta$ in the proof will be the same for each different $i$. Calling $\gamma=\gamma_i$,  the path $\lift P_j^k$ uses the edge $w^j\gamma \lift E$ if and only if $j=i$ for some $0\le i < k$.
\end{proof}

The $\gamma$ in the lemma above can be removed if we additionally assume $\lift P_j$ begins with $\lift E$:
\begin{lem}\label{Pbeginslem}
Let $V_j$ be a vertex of $E$ and assume that $E$ belongs to $w\in \pi$ and $\lift P_j$ begins with $\lift E$. If $\rho(V_j)=w^{1/n}$ then $\lift P_j^k$ contains the edge $w^i\lift E$ if and only if $0\le i < k$. 
\end{lem}
\begin{proof}
We will show that the word $\gamma$ in the above lemmas becomes trivial in this case.
Consider what happens in the proof of Lemma \ref{belongslem} if $\lift P_j$ (called $\lift P$ in that lemma) begins with $\lift E$. In that case we have $i=1$ in that proof, and so $\gamma = \sigma_1 = \theta_1$, and since the edge $E$ is taken to be an edge of the spanning tree for our coherent labeling, this word is trivial. 
\end{proof}

As in the previous section, in what follows we will choose some edge $E \in G$ with initial vertex $V_1$ and terminal vertex $V_2$. We will assume $V_1$ is a periodic point with period $m$ and rotation element $w_1^{1/m}$ and that $V_2$ is a periodic point with period $n$ and rotation element $w_2^{1/n}$. We will choose the edge $E$ to be part of the spanning tree of $G$. This means that  $\lift E$ has initial vertex $\lift V_1$ and terminal vertex  $\lift V_2$.

In the following subsections we continue our theme of generalizing the rotation interval, this time subject to conditions related to the belonging relation above. Generally, when $E$ belongs to one or both of $w_1$ and $w_2$ we will have a theorem. In the case where $E$ does not belong to either $w_1$ or $w_2$ no general result is possible. This was demonstrated in Example \ref{newedge}, where $E$ belongs to neither of $w_1$ or $w_2$ (since $E$ was not part of the original graph), but the only periodic points in $E$ are fixed.

\subsection{$E$ belongs to exactly one of $w_1$, $w_2$.}
\begin{thm}\label{belongstoone}
\begin{enumerate}
\item If $E$ belongs to $w_1$, but not to $w_2$, there exists  an initial word, $\gamma_1$, of $w_1$ such that for any rational $r$ satisfying $0 < p/q < 1/m$,  there are $p,q$ with $r=p/q$ and a periodic point in $E$ with rotation element $(w_1^p \gamma_1)^{1/q}$.

\item  If $E$ belongs to $w_2$ but not to $w_1$, there exists $\gamma_2$  an initial word of $w_2$ such that for any rational $s$ satisfying $0 < s< 1/n$, there are $p',q'$ with $s=p'/q'$ and a periodic point in $E$ with rotation element $(w_2^{p^\prime} \gamma_2)^{1/q^\prime}$.

\end{enumerate}
\end{thm}
\begin{proof}
We prove the first statement. 
Let $r=l/k$.
Since $E$ belongs to $w_1$, by Lemma \ref{Pbelongslem} there exists $\gamma_1$, an initial word of $w_1$, such that $ w_1^i \gamma_1 \lift E$ is in  $\lift P_1^{kn}$ for $0 \leq i < kn$. Since $E$ is not contained in $w_2$, $\lift P_2^{km}$ contains no edges that project onto $E$ apart from, possibly, the first and last edges. So either $\lift P_1^{-kn}\lift E \lift P_2^{km}$ is contracted or, if it is not, the only edge that projects onto $E$ that cancels is $\lift E$. Thus $\lift P_1^{-kn}\lift E \lift P_2^{km}$ must contain $w_1^i  \gamma_1   \lift E$ for $0 \le i < kn$. Thus by Lemma \ref{roteltsinpathlem} there must be a periodic point in $E$ with rotation element $(w_1^i\gamma_1)^{1/kmn}$ for $0 < i < kn$. 

Since $0<l/k<1/m$, we know that $0<lmn<kn$. Thus there is a periodic point with rotation element  $(w_1^{lmn}\gamma_1)^{1/kmn}$. Since $l/k=r$, letting $p=lmn$ and $q=kmn$ gives the result.
\end{proof}

If in the first case we also have that $\lift P_1$ begins with $\lift E$, then we see that $\gamma_1$ will be trivial by using Lemma \ref{Pbeginslem} in place of Lemma \ref{Pbelongslem}. Similarly, in the second case if $\lift P_2$ begins with $\lift E$, then $\gamma_2$ will be trivial.

\subsection{$E$ belongs to both $w_1$, $w_2$.}

\subsubsection{Neither or both of $\lift P_1$ and $\lift P_2$ begin with $\lift E$}

\begin{thm}\label{EbothWEneitherP}
If $E$ belongs to both $w_1$ and $w_2$ and neither $\lift P_1$ nor $\lift P_2$ begins with $\lift E$, then there exists $\gamma_1$ an initial word of $w_1$ and $\gamma_2$  an initial word of $w_2$ such that 
for any rationals $r$ and $s$ satisfying $0 < r < 1/m$, $0<s<1/n$ there $p,q,p',q'$  with $r=p/q$ and $s=p'/q'$ and periodic points in $E$ with rotation elements $(w_1^p \gamma_1)^{1/q}$ and $(w_2^{p^\prime} \gamma_2)^{1/q^\prime}$.
\end{thm}

\begin{proof}
Choose $k$ such that there exist integers $l$ and $l^\prime$ with $r=l/k$ and $s=l^\prime/k$.
Since neither $\lift P_1$ nor $\lift P_2$ begin with $\lift E$, the path $\lift P_1^{-kn}\lift E \lift P_2^{km}$ is fully contracted. Since both $w_1$ and $w_2$ contain $E$ there exist $\gamma_1$ and $\gamma_2$ such that $ w_1^i \gamma_1 \lift E$ is in  $\lift P_1^{kn}$ for $0 \leq i < kn$ and that $w_2^j\gamma_2  \lift E$ is in  $\lift P_2^{km}$ for $0 \leq j < km$.  Take $i=lmn$, $j=l^\prime mn$ and applying Lemma \ref{roteltsinpathlem} completes the proof with $p=lmn, p'=l'mn,$ and $q=q'=k$.
\end{proof}

\begin{thm}\label{EbothWbothP}
If $E$ belongs to both $w_1$ and $w_2$ and both $\lift P_1$ and $\lift P_2$ begin with $\lift E$, then for any rational numbers $r$ and $s$ satisfying $0 \leq r < 1/m$ and $0 \leq s <1/n$ there are periodic points in $E$ with rotation elements $w_1^r$ and $ w_2^s$.
\end{thm}

\begin{proof}
Since both $\lift P_1$ nor $\lift P_2$ begin with $\lift E$, the path $\lift P_1^{-kn}\lift E \lift P_2^{km}$ is fully contracted after contracting  $\lift E^{-1}\lift E \lift E^{-1}$ to $\lift E^{-1}$. Since $E$ belongs to both $w_1$ and $w_2$ and both $P_1$ and $P_2$ begin with $\lift E$, Lemma \ref{Pbeginslem} shows that $ w_1^i  \lift E$ is in  $\lift P_1^{kn}$ for $0 \leq i < kn$ and  $w_2^j \lift E$ is in  $\lift P_2^{km}$ for $0 \leq j < km$.  

As above, choose $k$ such that there exist integers $l$ and $l^\prime$ with $r=l/k$ and $s=l^\prime/k$. Then take $i=lmn$, $j=l^\prime mn$, and apply Lemma \ref{roteltsinpathlem}. We obtain rotation elements $w_1^{l/k}=w_1^r$ and $w_2^{l'/k} = w_2^s$.
\end{proof}

The last case that we need to consider is when exactly one of $\lift P_1$ or $\lift P_2$ begins with $\lift E$. This is the case when several edges can cancel when contracting $\lift P_1^{-1}\lift E \lift P_2$. In the next subsection we look at the maximum number of edges that can cancel. We will see that when $w_1$ and $w_2$ are powers of some common word $w$ we obtain a result analogous to degree one maps of the circle that is outlined in the remark below. When $w_1$ and $w_2$ are not powers of some common word, our result is similar to Theorem \ref{EbothWEneitherP}.

\begin{rem}
Suppose that we have a degree one map of the circle that has two periodic orbits $O_1$ and $O_2$, that the points in $O_1$ have rotation number $r$ and those in $O_2$ have rotation number $s$, then it is well known that for any rational $t$ between $r$ and $s$ there is a periodic point with rotation number $t$. We can translate this fact into our setting in the following way.

Consider the points in the two orbits to be vertices and, with this interpretation, the map can be considered to be a vertex  map on a graph. Let $a$ be the generator of the fundamental group. The points in $O_1$ will all have rotation elements of the form $a^r$ and those in $O_2$ will equal $a^s$. The result above then says for any rational $t$ between $r$ and $s$ there is a periodic point with rotation element $a^t$. 
\end{rem}

\subsubsection{$E$ belongs to $w$, and $w_1$ and $w_2$ are powers of a common word $w$}

Here we assume there is a word $w$ such that $w_1 = w^{k_1}$ and $w_2 = w^{k_2}$.

\begin{thm}\label{cfrotint}
Suppose that $E$ belongs to $w$. If $\rho(V_1)=w^{k_1/m}$ and $\rho(V_2)=w^{k_2/n}$, then for any rational number $r$ that lies between $k_1/m$ and $k_2/n$ there exists $p$ and $q$ with $p/q=r$ and a periodic point $x \in E$ with $\rho(x)=(w^{p}\gamma)^{1/q}$, where $\gamma$ is an initial word of $w$.
\end{thm} 

\begin{proof}
Assume without loss of generality that $k_1/m < k_2/n$. Take any natural numbers $p',q'$ with $p'/q'=r$, and we will look at the path $\lift P_1^{-q'k_2n} \lift E \lift P_2^{q'k_2m}$ and show that it must contain $w^{k_2p'mn}\gamma \lift E$. We do this by first showing that $\lift P_2^{q'k_2m}$ contains $w^{k_2p'mn}\gamma \lift E$ and then showing that $\lift P_1^{q'k_2n}$ does not. 

Consider $\lift P_2^{q'k_2m}$. Lemma \ref{Pbelongslem} tells us that there is some initial subword $\gamma_2$ of $w_2$ such that $\lift P_2^{q'k_2m}$ contains $w_2^j\gamma_2 \lift E$ whenever $0 \le j< q'k_2m$. Write $\gamma_2$ as $w^t \gamma$, where $\gamma$ is an initial word of $w$. Since $w=w_2^{k_2}$ it must be the case that $t\le k_2$. 
 
Writing $w_2$ and $\gamma_2$ in terms of $w$ tells us that $w^{k_2j+t}\gamma \lift E$ is contained in $\lift P_2^{q'k_2m}$ whenever $0 \le j< q'k_2m$. 

Since $t\le k_2$ we have $k_2(j-1)+t \le k_2j  \leq k_2j + t$. When $1 \leq j < q'k_2m$, both $w^{k_2(j-1)+t}\gamma \lift E$ and $w^{k_2j+t}\gamma \lift E$ are contained in $\lift P_2^{q'k_2m}$. The shortest path that contains both $w^{k_2(j-1)+t}\gamma \lift E$ and $w^{k_2j+t}\gamma \lift E$  contains  $w^{k_2j}\gamma \lift E$. Thus $w^{k_2j}\gamma \lift E$ is contained in $\lift P_2^{q'k_2m}$ for $1 \leq j< q'k_2m$. We now choose a particular value for $j$.

Since $p'/q' < k_2/n$, we have $p'mn < q'k_2m$. Let $j=p'mn$. The previous paragraph shows that $w^{k_2p'mn}\gamma \lift E$ belongs to $\lift P_2^{q'k_2m}$. 
 
We now consider $\lift P_1^{q'k_2n}$. It contains $w_1^i\gamma_1 \lift E$ for $0 \le i< q'k_2n$ for some initial word $\gamma_1$ of $w_1$. The words $w$ and $\gamma_1$ can be rewritten in terms of $w$. But note that, by Lemma \ref{Pbelongslem}, $\lift P_1^{q'k_2n}$ cannot contain any terms of the form $w^i \gamma\lift E$ for $i>q'k_2nk_1$. It is straightforward to check that $k_1/m < p'/q'$ implies that $k_2p'mn>q'k_2nk_1$. So $w^{k_2p'mn}\gamma \lift E$ belongs to $\lift P_2^{q'k_2m}$ but not to $\lift P_1^{q'k_2n}$. Consequently it must belong to $\lift P_1^{-q'k_2n} \lift E \lift P_2^{q'k_2m}$.
 So by  Lemma \ref{roteltsinpathlem} there exists a periodic point $x \in E$ with $\rho(x)=(w^{k_2p'mn}\gamma)^{1/k_2q'mn}$. Letting $p=k_2p'mn$ and $q=k_2q'mn$ gives the result.
\end{proof}

\begin{rem} The hypotheses of the above lemma do not take into account whether or not the paths $\lift P_1$ and $\lift P_2$ begin with $\lift E$. However, the proof considers the maximum possible number of edges in $\lift P_1^{-1} \lift E \lift P_2$ canceling. In the cases when both $\lift P_1$ and $\lift P_2$ begin with $\lift E$ or neither begin with $\lift E$, the contracted path $\lift P_1^{-1} \lift E \lift P_2$ still contains $\lift E$ and so in these cases the conclusions of Theorems \ref{EbothWbothP} and \ref{EbothWEneitherP} give stronger results. Consequently, the above lemma only provides useful information when exactly one of $\lift P_1$ and $\lift P_2$ begin with $\lift E$. 
\end{rem}

\subsubsection{$E$ belongs to both $w_1$ and $w_2$, but $w_1$ and $w_1$ are not powers of a common word}

\begin{thm} \label{belongsnotpowers}
Suppose $E$ belongs to both $w_1$ and $w_2$ and $w_1$ and $w_2$ are not powers of a common word. Then there exists $\gamma_1$ an initial word of $w_1$ and $\gamma_2$  an initial word of $w_2$ such that 
for any rationals $r$ and $s$ satisfying $0 < r < 1/m$, $0<s<1/n$ there exist $p,q,p',q'$ and periodic points in $E$ with rotation elements $(w_1^p \gamma_1)^{1/q}$ and $(w_2^{p^\prime} \gamma_2)^{1/q^\prime}$, where $r=p/q$ and 
$s=p^\prime/q^\prime$.\end{thm}

\begin{proof}

For any $k>0$,  by Lemma \ref{Pbelongslem} there are initial words $\gamma_1, \gamma_2$ of $w_1, w_2$ such that $ w_1^i \gamma_1 \lift E$ is in  $\lift P_1^{km}$ for $0 \leq i < km$ and that $w_2^j\gamma_2 \lift E$ is in  $\lift P_2^{kn}$ for $0 \leq j < kn$. Since $w_1$ and $w_2$ are not powers of some common word, we will have $w_1^i\gamma_1 \neq w_2^j\gamma_2$ when $i$ and $j$ are sufficiently large. Thus there is some $N$ such that $w_1^i  \gamma_1 \lift E$ is in $\lift P_1^{km}$ but not in $\lift P_2^{kn}$ for $N\le i < km$ and $w_2^j\gamma_2\lift E$ is in $\lift P_2^{kn}$ for $N \le j < kn$, and we may choose $k$ sufficiently large so that $N < km$ and $N < kn$.

Since $w_1^i\gamma_1\lift E$ is in $\lift P_1^{kn}$ but not $\lift P_2^{km}$ for $N \le i < km$, this edge does not cancel in the contraction of $\lift P_1^{kn}\lift E \lift P_2^{km}$. Thus by Lemma \ref{roteltsinpathlem} there is a periodic point $x$ with rotation element $\rho(x) = (w_1^i\gamma_1)^{1/kmn}$. Now let $\bar p, \bar q$ be large numbers with $r = \bar p/ \bar q$. The number $k$, which was chosen to be any sufficiently large number, can be chosen to be a multiple of $\bar q$, say $k=l\bar q$. 

Then let $i = l\bar pmn$, being sure that $\bar p$ is sufficiently large so that $N<i$. Since $\bar p / \bar q < 1/n$ it is easy to see that $i=l\bar pmn < km$, and so we have $\rho(x) = (w_1^{l\bar p mn}\gamma_1)^{1/l\bar qmn}$. Letting $p=l\bar pmn$ and $q=l\bar qmn$ gives the result.

Repeating the arguments in the above two paragraphs using $w_2^j$ in place of $w_1^i$ gives a periodic point with rotation element $(w_2^{p'} \gamma_2)^{1/q'}$.
\end{proof}

\begin{rem}

The comments contained in the previous remark are also relevant here. In the cases when neither $\lift P_1$ and $\lift P_2$ begin with $\lift E$ or both begin with $\lift E$, the conclusions of Theorems \ref{EbothWEneitherP} and \ref{EbothWbothP} give stronger results. So the lemma only provides useful information when exactly one of $\lift P_1$ and $\lift P_2$ begin with $\lift E$. As noted earlier, if $\lift P_1$ begins with $\lift E$ then $\gamma_1$ is trivial, and if $\lift P_2$ begins with $\lift E$ then $\gamma_2$ is trivial.

\end{rem}

\section{The vertices form one periodic orbit}

In \cite{B} the rotation matrix was defined. This matrix was defined in terms of homology. If instead of defining the rotation elements in terms of homotopy we use homology, then by Lemma \ref{orbitindep} there will be no distinction between the rotation elements of different vertices that belong to the same orbit. In \cite{B} it was shown that if the vertices of the graph form one periodic orbit then the rotation matrix is the zero matrix. This means that using the rotation matrix, all the vertices have the same rotation element and no periodic points with different rotation elements can be detected. The following theorem shows how defining rotation elements using homotopy give much more information. In particular, if the vertices form one periodic orbit it is the case that the map can have infinitely many periodic orbits with different rotation elements.

First we give an example of a graph and a linear map with the properties that the vertices form one periodic orbit and the map has infinitely many periodic points with distinct rotation elements. We take Example \ref{notcont}, but instead of defining $V_1$ goes to $V_2$ along $E_2$ we define the path to be $E_2E_3E_1E_2$. Everything else is defined as in Example \ref{notcont}. The rotation element of $V_2$ is now $(ba^2)^{1/5}$ and the rotation element of $V_1$ is $(a^2b)^{1/5}$. Since $E_2$ belongs to $a^2b$ and $\lift P_1^5$ begins with $\lift E_2$, $\lift P_2^5$ does not begin with $\lift E_2$, then Theorem \ref{belongsnotpowers} shows that there must be infinitely many periodic points with different rotation elements in $E_2$.

\begin{lem}
Suppose that the fundamental group of $G$ has at least two generators. Let $f:G \to G$ be a map that is homotopic to the identity and has the property that the vertices form one periodic orbit. Then there is an edge that has terminal vertices with distinct rotation elements.
\end{lem}

\begin{proof}
Since the map is homotopic to the identity and the fundamental group has at least two elements it follows from the Lefschetz Fixed Point Theorem that one of the edges under the linearization must contain a fixed point, and moreover the linearization has positive slope at the fixed point (see \cite{B4}). This fixed point will be on the interior of some edge, since the vertices cannot be fixed. Let $E$ denote this edge that contains a fixed point under the linearized map, and $V_1$ and $V_2$ the initial and terminal vertices. Choose a spanning tree containing $E$, and use this tree to give a coherent labeling of the universal cover. Let $n$ be the period of $V_1$ and $V_2$, and so there are words $w_1$ and $w_2$ with $\rho(V_1)=w_1^{1/n}$ and $\rho(V_2)=w_2^{1/n}$. We will show that $w_1 \neq w_2$.

Since the linearization has positive slope at the fixed point, neither of the paths $\lift P_1$ or $\lift P_2$ can begin with $\lift E$. So the contracted path from $w_1 \lift V_1$ to $w_2 \lift V_2$ must contain $\lift E$. This implies that $w_1 \neq w_2$, because if $w_1 = w_2$ the contracted path from $w_1 \lift V_1$ to $w_2 \lift V_2$ would just contain the single edge $w_1 \lift E$.
\end{proof}

\begin{thm} Suppose that the fundamental group of $G$ has at least two generators. Let $f:G \to G$ be a map that is homotopic to the identity and has the property that the vertices form one periodic orbit. Let $w^{1/v}$ denote the rotation element of one of the vertices. If every edge is belongs to $w$, then $f$ will have infinitely many periodic orbits with different rotation elements.
\end{thm}

\begin{proof}
The previous lemma shows that there must be an edge $E$ that has vertices with distinct rotation elements. Use this edge to form a spanning tree and then to coherently label the universal cover. 
When the construction in the lemma is followed, neither $\lift P_1$ nor $\lift P_2$ begins with $\lift E$. Thus the hypothesis of Theorem \ref{EbothWEneitherP} is satisfied, and its conclusion gives infinitely many periodic points with distinct rotation elements.
\end{proof}

\end{document}